\documentclass[12pt]{amsart}
\usepackage{bbold}
\usepackage{amsfonts}
\usepackage{amsmath}
\usepackage{amssymb}
\usepackage{amsthm}
\usepackage{setspace}
\usepackage{fullpage}
\usepackage{hyperref}

\newtheorem{thm}{Theorem}[section]
\newtheorem{prop}[thm]{Proposition}
\newtheorem{lem}[thm]{Lemma}
\newtheorem{conj}[thm]{Conjecture}

\newtheorem{rem}[thm]{Remark}
\newtheorem*{defn}{Definition}

\newcommand{\os}{Odlyzko and Stanley}

\newcommand{\iss}{independent Stanley sequence}

\newcommand{\lb}{\lambda}
\newcommand{\om}{\omega}
\newcommand{\al}{\alpha}
\newcommand{\ep}{\epsilon}
\newcommand{\kp}{\kappa}
\newcommand{\bs}{\backslash}

\title{On the growth of Stanley sequences}
\author{David Rolnick}
\thanks{Massachusetts Institute of Technology, Cambridge MA. Email: \texttt{drolnick@math.mit.edu}}
\author{Praveen S. Venkataramana}
\thanks{Massachusetts Institute of Technology, Cambridge MA. Email: \texttt{venkap@mit.edu}}

\begin{document}

\begin{abstract}
A set is said to be \emph{3-free} if no three elements form an arithmetic progression. Given a 3-free set $A$ of integers $0=a_0<a_1<\cdots<a_t$, the \emph{Stanley sequence} $S(A)=\{a_n\}$ is defined using the greedy algorithm: For each successive $n>t$, we pick the smallest possible $a_n$ so that $\{a_0,a_1,\ldots,a_n\}$ is 3-free and increasing. Work by Odlyzko and Stanley indicates that Stanley sequences may be divided into two classes. Sequences of Type 1 are highly structured and satisfy $\al n^{\log_2 3}/2\le a_n\le \al n^{\log_2 3}$, for some constant $\al$, while those of Type 2 are chaotic and satisfy $\Theta(n^2/\log n)$. In this paper, we consider the possible values for $\alpha$ in the growth of Type 1 Stanley sequences. Whereas Odlyzko and Stanley assumed $\alpha=1$, we show that $\alpha$ can be any rational number which is at least 1 and for which the denominator, in lowest terms, is a power of 3.
\end{abstract}

\maketitle

\section{Introduction}

Sets without arithmetic progressions are a perennially interesting topic in mathematics.  In 1953, Roth \cite{roth} proved that any infinite set of integers with linear density must contain an arithmetic progression (AP). Sanders \cite{sanders} recently improved this result to density $\Omega(n/\log^{1-o(1)} n)$.

A set without 3-term APs is called a \emph{3-free set}. Odlyzko and Stanley introduced the natural idea of constructing 3-free sets by the greedy algorithm, starting with some finite set of elements. Suppose that $A$ is a finite 3-free set of nonnegative integers $0=a_0 < a_1 < \cdots < a_t$. The \emph{Stanley sequence of $A$}, denoted $S(A) = \{a_n\}$ of $A$, is the infinite sequence of nonnegative integers defined recursively such that for $n> t$, we pick $a_n>a_{n-1}$ to be the smallest integer for which the set $\{a_0,a_1,\ldots,a_n\}$ is 3-free.  For simplicity we will often denote $S(\{a_0,\ldots,a_t\})$ by $S(a_0,\ldots,a_t)$.

The simplest Stanley sequence is $S(0)=0,1,3,4,9,10,12,13,27,\ldots$. It is simple to show that the $n$th term of this sequence is the number obtained by writing $n$ in binary and interpreting it in ternary. In particular, the $2^k$th term is $3^k$. \os\ \cite{stanley} found equally explicit expressions, involving ternary digits, for $S(0,3^m)$ and $S(0,2\cdot 3^m)$.

The asymptotic growth of Stanley sequences remains poorly understood. It has been conjectured that Stanley sequences fall into two classes, of which the first class are highly structured and grow slowly, while the second class appear chaotic and grow more quickly.

\begin{conj}[based on work by Odlyzko and Stanley \cite{stanley}]
Every Stanley sequence $S(A)=\{a_n\}$ follows one of two patterns of asymptotic growth:

\emph{Type 1:} $\al n^{\log_2 3}/2\le a_n\le \al n^{\log_2 3}$, where $\al$ is a constant, or

\emph{Type 2:} $a_n=\Theta(n^2/\log n)$.

\label{conj:growth}
\end{conj}

Odlyzko and Stanley \cite{stanley} considered Type 1 sequences only in the case $\al=1$. They showed that $S(0)$, $S(0,3^m)$, and $S(0,2\cdot 3^m)$ all follow this asymptotic growth pattern, with $a_{2^k}=3^k$ for large enough $k$. Likewise, Erd\H{o}s et al.~\cite{erdos} found that the sequence $S(0,1,4)$ satisfies $a_{2^k}=3^k + 2^{k-1}$ if $k\ge 2$. However, Rolnick \cite{rolnick} demonstrated that many Stanley sequences follow Type 1 growth for other values of $\al$, including the sequence $S(0,1,7)$, for which $a_{2^k}=(10/9)\cdot 3^k$, thus $\al=10/9$. Given a Type 1 sequence, we say that $\al$ is its \emph{scaling factor}.

By contrast, no Stanley sequence has been proven conclusively to satisfy Type 2 growth. Empirical observations by Lindhurst \cite{lindhurst} suggest that the sequence $S(0,4)$ does indeed follow this pattern of growth. A probabilistic argument by Odlyzko and Stanley \cite{stanley} indicates that a ``random'' Stanley sequence should be of Type 2, but does not prove that any actual Stanley sequence is of this form.

In a recent paper, Moy \cite{moy} solved a problem posed by Erd\H{o}s et al.~\cite{erdos}, showing that every Stanley sequence $\{a_n\}$ satisfies $a_n\le n^2/(2+\ep)$ for large enough $n$. Another question in \cite{erdos} was resolved by Savchev and Chen \cite{savchev}, who constructed a sequence $\{a_n\}$ for which $\lim_{n\to \infty} (a_{n+1}-a_n)=\infty$ and such that $\{a_n\}$ is a \emph{maximal} 3-free set, that is, a 3-free set that is not a proper subset of any other 3-free set. A related problem of \cite{erdos} remains open, finding a Stanley sequence $\{a_n\}$ for which $\lim_{n\to \infty} (a_{n+1}-a_n)=\infty$.

Results by Rolnick \cite{rolnick} imply that the scaling factor $\al$ of a Type 1 Stanley sequence may be arbitrarily high. In this paper, we prove a stronger result. Suppose that $\al$ is any rational number for which the denominator in lowest terms is a power of 3. Then, there exists a Stanley sequence with scaling factor $\al$.

We also consider the \emph{repeat factor}, which intuitively is the value $a_n$ at which a sequence $S(A)$ begins to exhibit Type 1 behavior. We demonstrate that every sufficiently large integer is the repeat factor for some Type 1 sequence.

\section{Preliminaries}

Building on the work of \os\ \cite{stanley}, Rolnick \cite{rolnick} introduced the notion of independent Stanley sequences. A related notion, that of dependent sequences, will not be relevant to this paper.  It can be shown that a Stanley sequence is of Type 1 if it is independent or dependent, and Rolnick conjectured that the converse also holds.

\begin{defn}
A Stanley sequence $S(A)=\{a_n\}$ is \emph{independent} if there exist constants $\lb=\lb(A)$ and $\kappa=\kappa(A)$ such that for all $k\ge \kappa$,

\begin{itemize}
\item$a_{2^k + i} = a_{2^k} + a_i$ when $0\le i < 2^k$, and
\item$a_{2^k} = 2a_{2^k - 1} + 1 - \lambda$.
\end{itemize}
\end{defn}

It is simple to verify that the sequences $S(0)$, $S(0,3^k)$, and $S(0,2\cdot 3^k)$ are independent. Rolnick identified several more general classes of independent Stanley sequences (see Theorems 1.2 and 1.4 of \cite{rolnick}).

When $k\ge \kappa$, it follows from the definition of an independent Stanley sequence that

$$a_{2^{k+1}} = 2a_{2^{k+1}-1}-\lambda+1 = 2(a_{2^k}+a_{2^k-1})-\lambda+1 = 3a_{2^k}.$$

Hence, for an independent Stanley sequence $S(A)$ there exists a positive number $\al=\al(A)$ such that for sufficiently large $k$, $$a_{2^k} = \al\cdot 3^k.$$We say that $\al$ is the \emph{scaling factor} of $S(A)$. The following proposition shows that each term of $S(A)$ is approximately $\al$ times the corresponding term of $S(0)$. In particular, this implies that independent Stanley sequences follow Type 1 growth.

\begin{prop}
\label{prop:scaling}
Let $S(A)=\{a_n\}$ be a Stanley sequence. Then, $S(A)$ is independent if and only if for every $n$,

\begin{equation}a_n = \al s_n + b_n,\label{eq:scaling}\end{equation}

\noindent where $\al$ is a constant, $S(0)=\{s_n\}$, and $\{b_n\}$ is a periodic sequence with period $2^\kappa$. If (\ref{eq:scaling}) holds, then $\al$ is the scaling factor of $S(A)$ and $\kappa=\kappa(A)$.
\end{prop}
\begin{proof}
We first prove that if $S(A$) is independent then (\ref{eq:scaling}) holds.

Let $\{b_n\}$ be the sequence in which the values $\{a_0-\al s_0,a_1-\al s_1,\ldots,a_{2^\kappa-1}-\al s_{2^\kappa-1}\}$ repeat periodically. Pick $k\ge \kappa$, and suppose towards induction that $a_n=\al s_n+b_n$ for all $n<2^k$. The base case of $k=\kappa$ holds from the definition of the sequence $\{b_n\}$.

For each $n<2^k$, we have \begin{align*}a_{n+2^k}&=a_{2^k}+a_n\\ &=\al s_{2^k}+(\al s_n+b_n)\\ &=\al(s_{2^k}+s_n)+b_n\\ &=\al s_{2^k+n}+b_{2^k+n}.\end{align*}The last step follows since $b_n$ is periodic. Therefore, $a_n=\al s_n+b_n$ for all $n<2^{k+1}$, completing the induction. One last check remains: The period of $\{b_n\}$ might be some proper divisor of $2^\kappa$. However, it is easily verified that this implies $\kp(A)<\kp$; hence the period of $\{b_n\}$ is indeed $S^\kp$.

Now, assume that (\ref{eq:scaling}) holds. Observe that for $k\ge \kappa$, and $0\le i<2^k$,\begin{align*}a_{2^k+i}&=\al s_{2^k+i}+b_{2^k+i}\\ &=\al s_{2^k}+\al s_i+b_{2^k}+b_i=\\ &=a_{2^k}+a_i\\ \text{and}\hskip .25 in a_{2^k}&=\al s_{2^k}+b_{2^k}\\ &=\al(2s_{2^k-1}+1)+b_{2^\kappa}\\ &=2\al(s_{2^k-1}+b_{2^k-1})-(2\al b_{2^\kappa-1}-b_{2^\kappa}-\al+1)+1.\end{align*}Thus, we see that $S(A)$ is independent with $\lb(A)=2\al b_{2^\kappa-1}-b_{2^\kappa}-\al+1$. Since the period of $\{b_n\}$ is $2^\kappa$ and not a proper divisor, it follows that $\kappa$ is the minimum $k$ for which the independence conditions hold; hence $\kappa=\kappa(A)$.
\end{proof}

Given an independent Stanley sequence $S(A)$, we define the \emph{repeat factor} $\rho(A):=a_{2^\kappa}=\al\cdot 3^\kappa$. Thus,

\begin{itemize}
\item{$\rho(\{0\}) = 1$;}
\item{$\rho(\{0,3^{k}\}) = \rho(\{0,2\cdot 3^{k-1}\}) = 3^{k+1}$.}
\end{itemize}

\begin{prop}
\label{prop:repeat}
A Stanley sequence $S(A)=\{a_n\}$ is independent if and only if

\begin{equation}\{a_n\}=\left\{\rho x+y\mid x\in S(0),\, y\in \{a_0,a_1,\ldots,a_{2^\kappa-1}\}\right\}\label{eq:repeat}\end{equation}

\noindent where $\rho$ and $\kappa$ are constants. If (\ref{eq:repeat}) holds, then $\rho$ is the repeat factor of $S(A)$ and the minimum value of $\kappa$ equals $\kappa(A)$.
\end{prop}

The proof of Proposition \ref{prop:repeat} follows from a straightforward induction argument similar to that of Proposition \ref{prop:scaling}.

Just as the scaling factor of an independent Stanley sequence measures how the sequence behaves asymptotically, so the repeat factor measures how fast the sequence converges to its asymptotic behavior. In this paper, we consider the set of rational numbers that are scaling factors and the set of integers that are repeat factors.

We define a \emph{triadic number} to be a rational number for which the denominator, in lowest terms, is a power of $3$.

\begin{thm}
\label{thm:main} (i) Every sufficiently large integer $\rho$ is a repeat factor. (ii) Every triadic number $\al \ge 1$ is a scaling factor.
\end{thm}

\section{Main result}

To prove Theorem \ref{thm:main}, we develop a construction for independent Stanley sequences that allows us carefully to control the scaling factor and repeat factor. We begin with several lemmas.

We say that an integer $x$ is \emph{covered} by a set $S$ if there is a 3-term AP of the form $y,z,x$ with $y<z<x$ and $y,z\in S$. Likewise, we say an integer $x$ is \emph{jointly covered} by sets $S$ and $T$ if there is a 3-term AP of the form $y,z,x$ with $y<z<x$, such that $y\in S$ and $z\in T$. Given a Stanley sequence $S(A)$, let $O(A)$ be the set of nonnegative integers neither in $S(A)$ nor covered by it. By the definition of a Stanley sequence, $O(A)$ must be a finite set.  We let $\omega(A)$ be the maximum element of $O(A)$.

\begin{lem}[Rolnick \cite{rolnick}]
\label{lem:kappa}
Let $S(A) = \{a_n\}$ be a Stanley sequence. Suppose that there are integers $\lambda$ and $k$ satisfying $a_{2^k-1} \ge \lambda+\omega(A)$ such that 

\begin{itemize}
\item{$a_{2^k + i} = a_{2^k} + a_i$ for all $0\le i < 2^k$, and}
\item{$a_{2^k} = 2a_{2^k - 1} - \lambda + 1 $.}
\end{itemize}

Then, $S(A)$ is independent with $\kappa(A)\le k$ and $\lambda(A)=\lambda$.
\end{lem}

If $x$ is an integer and $S$ is a set, we will use the notation $S+x$ to denote the set $\{y+x\mid y\in S\}$. The next lemma is based on methods used by Rolnick \cite{rolnick}.

\begin{lem}
\label{lem:cover}
Let $S(A)$ be an independent Stanley sequence. For some $k\ge \kappa$, set $c=a_{2^k}$ and let $A_k=\{a_0,a_1,\ldots,a_{2^k-1}\}$.  Suppose that $a_{2^k-1} \ge \lambda(A)+\omega(A)$. Then, the following statements hold for all integers $x,y$ such that $x<y$.

\begin{enumerate}
\renewcommand{\theenumi}{\alph{enumi}}
\item The set $A_k+x$ covers $$[x,c+x)\bs ((A_k\cup O(A))+x)\cup (O(A)+c+x).$$

\item The sets $A_k+x$ and $A_k+y$ jointly cover the set $$[2y-x,c+2y-x)\bs (O(A)+2y-x)\cup (O(A)+c+2y-x).$$

\item The set $(A_k+x)\cup (A_k+c+x)$ covers $$[x,3c+x)\bs ((A_k\cup (A_k+c)\cup O(A))+x)\cup (O(A)+3c+x).$$

\item The sets $(A_k+x)\cup (A_k+c+x)$ and $(A_k+y)\cup (A_k+c+y)$ jointly cover the set $$[2y-x,3c+2y-x)\bs (O(A)+2y-x)\cup (O(A)+3c+2y-x).$$

\item The set $(A_k+x)\cup (A_k+c+x)\cup (A_k+3c+x)\cup (A_k+4c+x)$ covers the set $$[x,9c+x)\bs ((A_k\cup (A_k+c)\cup(A_k+3c)\cup(A_k+4)\cup O(A))+x)\cup (O(A)+9c+x).$$
\end{enumerate}
\end{lem}
\begin{proof}
We first prove part (a). Observe that the set $A_k$ must cover every integer in $[0,c)\bs (A_k\cup O(A))$ because these integers are assumed to be covered by $S(A)$.  Hence, $A_k+x$ covers $[x,c+x)\bs ((A_k\cup O(A))+x)$.  Now consider an integer $z$ in $O(A)+c$. Since $z\not\in S(A)\cup O(A)$, we know that $S(A)$ covers $z$, hence one of the following must be true: (i) $z$ is covered by $A_k$, (ii) $z$ is jointly covered by $A_k$ and $A_k+c$, or (iii) $z$ is covered by $A_k+c$. Case (iii) is impossible because $z\in O(A)+c$ and by definition no element of $O(A)$ is covered by $A_k$. Case (ii) is almost impossible, since the smallest integer jointly covered by $A_k$ and $A_k+s$ is \begin{align*}2a_{2^k}-a_{2^k-1}&=c+(a_{2^k}-a_{2^k-1})\\ &=c+a_{2^k-1}-\lb(A)+1\\ &\ge c+(\lb(A)+\om(A))-\lb(A)+1\\ &>c+\om(A).\end{align*}  We conclude that $z$ must be covered by $A_k$, so $O(A)+c$ is covered by $A_k$ and hence $O(A)+c+x$ is covered by $A_k+x$.

We now prove part (b). Note that if $z$ is covered by $A_k$, then $z+2y-x$ is jointly covered by $A_k+x$ and $A_k+y$. Applying part (a), then, we conclude that $A_k+x$ and $A_k+y$ jointly cover the set $[2y-x,c+2y-x)\backslash ((A_k\cup O(A))+2y-x)\cup (O(A)+c+2y-x)$. Furthermore, $A_k+x$ and $A_k+y$ jointly cover $A_k+2y-x$ because, for each $a\in A_k$, the integers $a+x,a+y,a+2y-x$ form an AP.

Parts (c) and (d) follow from parts (a) and (b), respectively, by setting $k\leftarrow k+1$.  Part (e) follows from part (a) by setting $k\leftarrow k+2$.
\end{proof}

The following proposition is the driving force behind the proof of Theorem \ref{thm:main}.

\begin{prop}
\label{prop:main}
Let $S(A) = \{a_n\}$ be an \iss, with $k>\kappa(A)$. Suppose that $a_{2^k-1} \ge \lambda(A)+\omega(A)$.  Pick $d$ any integer such that $\omega(A)<d\le a_{2^k}-\lb(A)$ and set $$A^d_k=A\cup (a_{2^k}+ A)\cup (7a_{2^k} - d + A)\cup (8a_{2^k} - d + A).$$ Then, $S(A^d_k)=\{a'_n\}$ is independent, with \begin{align*}\rho(A^d_k)&=10a_{2^k}-d\\ \alpha(A^d_k)&=\frac{10\alpha(A)}{9}-\frac{d}{3^{k+2}}.\end{align*}
\end{prop}

Before proving the proposition, we provide a motivating result from \cite{rolnick}.

\begin{prop}[Rolnick \cite{rolnick}]
\label{prop:product}
Suppose that $S(A)=\{a_n\}$ and $S(B)=\{b_n\}$ are Stanley sequences, and that $k\ge\kappa(A)$. Let $A^{\ast}=\{a_0,a_1,\ldots,a_{2^k-1}\}$ and define $$A\otimes_k B=\{a_{2^k}b+a\mid  a\in A^{\ast},b\in B\}.$$Then, if $S(A)$ and $S(B)$ are independent, $S(A\otimes_k B)$ is an independent sequence having description $$S(A\otimes_k B)=\{a_{2^k}b+a\mid  a\in A^{\ast},b\in S(B)\}.$$
\end{prop}

\begin{rem}
Proposition \ref{prop:repeat} implies that a Stanley sequence $S(A)$ satisfies $S(A\otimes_\kappa \{0\})=S(A)$ (for some $\kappa$) if and only if $S(A)$ is independent.
\end{rem}

\begin{rem}
It is readily verified that the scaling factor of $S(A\otimes_k B)$ is simply the product of the scaling factors of $S(A)$ and $S(B)$. Taking $A=\{0,1,7\}$ and $B_0=\{0,1,7\}$, so that $\alpha(A)=10/9$, it follows that iterated products $S(B_n)=S(A\otimes_k B_{n-1})$ satisfy $\alpha(B_n)=(10/9)^n$. Hence, the scaling factor can be made arbitrarily large. Theorem \ref{thm:main} clearly proves a much stronger result.
\end{rem}

In the light of Proposition \ref{prop:product}, Proposition \ref{prop:main} may be seen as defining Stanley sequences that are in some sense ``intermediate'' between \begin{align*}S(A\otimes_k \{0,1,6,7\})&=S\left(A\cup (A+a_{2^k})\cup (A+6a_{2^k})\cup (A+7a_{2^k})\right)\\ \text{and}\hskip .1 in S(A\otimes_k \{0,1,7,8\})&=S\left(A\cup (A+a_{2^k})\cup (A+7a_{2^k})\cup (A+8a_{2^k})\right).\end{align*}

\begin{proof}[Proof of Proposition \ref{prop:main}]

Set $\lambda:=\lambda(A)$, $\omega:=\omega(A)$, $b:=a_{2^k-1}$, and $c:=a_{2^k}$, and $A^\ast=\{a_0,a_1,\ldots,a_{2^k-1}\}$. We define the following sets.

\begin{align*}B&:= A^\ast\cup (A^\ast+c)\\
C&:=(A^\ast+7c-d)\cup (A^\ast+8c-d)\\
D&:=(A^\ast+10c-d)\cup(A^\ast+11c-d)\\
E&:=(A^\ast+17c-2d)\cup(A^\ast+18c-2d)\\
F&:=(A^\ast+30c-3d)\cup(A^\ast+31c-3d)\\
G&:=(A^\ast+37c-4d)\cup(A^\ast+38c-4d)\\
H&:=(A^\ast+40c-4d)\cup(A^\ast+41c-4d)\\
I&:=(A^\ast+47c-5d)\cup(A^\ast+48c-5d)\\
J&:=B\cup C\cup D\cup E\cup F\cup G\cup H\cup I
\end{align*}

Our approach is as follows. We prove that (i)  $J$ is 3-free and (ii) the set $J$ covers all integers between $\max(C)=b+8c-d$ and $\max(I)=b+48c-5d$, with the exception of $J$ itself. This implies that $J$ is a prefix of $S(A^d_k)$. We now may apply Lemma \ref{lem:kappa} to prove that $S(A^d_k)=\{a'_n\}$ is independent. In order to apply this lemma, however, we require the condition $a'_{2^{k+3}-1}\ge \lb(A^d_k)+\om(A^d_k)$. (This is the reason we must consider such a large prefix subsequence of $S(A^d_k)$.) We may verify this condition as follows:

\begin{align*}\lb(A^d_k)&= 2(b+8c-d)-(10c-d)+1\\ &=2b+6c-d+1\\ &<8c-d\\ \om(A^d_k)&=\om(A)+8c-d\\ &<b+8c-d.\end{align*}

Hence,
\begin{align*}a'_{2^{k+3}-1}&=b+18c+2d\\ &>(8c-d)+(b+8c-d)\\ &>\lb(A^d_k)+\om(A^d_k).\end{align*}

We now show that $J$ is $3$-free. Suppose towards contradiction that $x,y,z\in J$ form an AP with $x<y<z$. Observe that $J$ reduces modulo $10c-d$ to $B\cup C$: $$J=(B\cup C)\cup (B\cup C+(10c-d))\cup (B\cup C+3(10c-d))\cup (B\cup C+4(10c-d)).$$There is no 3-term AP in the set $\{w,w+(10c-d),w+3(10c-d),w+4(10c-d)\}$ for any value of $w$; hence, $x$ and $y$ must be distinct modulo $10c-d$.

Notice that $C\cup D\equiv B\cup C\pmod{10c-d}$. Let $x',y'\in C\cup D$ be congruent, respectively, to $x,y$ modulo $10c-d$, and let $z'=2y'-x'$ so that $x',y',z'$ form a 3-term AP (possibly decreasing). Because $2y'-x'\equiv z\pmod{10c-d}$, we know that $z'\in B\cup C\cup D\cup E$. Since $7c-d\le x',y'\le b+11c-d$, we see that $$-b+3c-d\le 2y'-x'\le 2b+15c-d.$$Because $d\le c-\lb$, we conclude that\begin{align*}z'&\ge -b+3c-d=b-\lb+1+2c-d\ge b+c+1,\\ \text{and}\hskip .25 in z'&\le 2b+15c-d=16c+\lb-1-d\le 17c-2d-1.\end{align*}Hence, $z'$ cannot be in $B$ or $E$, so $z'\in C\cup D$.  But $C\cup D$ is 3-free since $$C\cup D=\{a_n+7c-d\mid 0\le n<2^{k+2}\},$$and we know that $\{a_n\}$ is 3-free. Therefore, $x',y',z'$ cannot be an AP, a contradiction.  We conclude that $J$ is 3-free, as desired.

We now use repeated applications of Lemma \ref{lem:cover} to prove that the set $J$ covers all elements of $[b+8c-d,b+48c-5d]\bs J$.

By part (e) of Lemma \ref{lem:cover}, $C\cup D$ covers the set \begin{equation}[7c-d,16c-d)\bs (C\cup D\cup(O(A)+7c-d))\cup (O(A)+16c-d).\label{one}\end{equation}

By part (d), $B$ and $C$ jointly cover \begin{equation}[14c-2d,17c-2d)\bs (O(A)+14c-2d)\cup (O(A)+17c-2d).\label{two}\end{equation}

By part (c), $E$ covers \begin{equation}[17c-2d,20c-2d)\bs (E\cup(O(A)+17c-2d))\cup (O(A)+20c-2d).\label{three}\end{equation}

By part (d), $B$ and $D$ jointly cover \begin{equation}[20c-2d,23c-2d)\bs (O(A)+20c-2d)\cup (O(A)+23c-2d).\label{four}\end{equation}

By part (b), $(A^\ast+11c-d)$ and $(A^\ast+17c-d)$ jointly cover \begin{equation}[23c-3d,24c-3d)\bs (O(A)+23c-3d)\cup (O(A)+24c-3d).\label{five}\end{equation}

By part (d), $D$ and $E$ jointly cover \begin{equation}[24c-3d,27c-3d)\bs (O(A)+24c-3d)\cup (O(A)+27c-3d).\label{six}\end{equation}

By part (d), $C$ and $E$ jointly cover \begin{equation}[27c-3d,30c-3d)\bs (O(A)+27c-3d)\cup (O(A)+30c-3d).\label{seven}\end{equation}

By part (c), $F$ covers \begin{equation}[30c-3d,33c-3d)\bs (F\cup(O(A)+30c-3d))\cup (O(A)+33c-3d).\label{eight}\end{equation}

By part (d), $E$ and $F$ jointly cover \begin{equation}[33c-4d,36c-4d)\bs (O(A)+33c-4d)\cup (O(A)+36c-4d).\label{nine}\end{equation}

By part (d), $B$ and $E$ jointly cover \begin{equation}[34c-4d,37c-4d)\bs (O(A)+34c-4d)\cup (O(A)+37c-4d).\label{ten}\end{equation}

By part (e), $G\cup H$ covers \begin{equation}[37c-4d,46c-4d)\bs (G\cup H\cup(O(A)+37c-4d))\cup (O(A)+46c-4d).\label{eleven}\end{equation}

By part (d), $F$ and $G$ jointly cover \begin{equation}[44c-5d,47c-5d)\bs (O(A)+44c-5d)\cup (O(A)+47c-5d).\label{twelve}\end{equation}

By part (c), $I$ covers \begin{equation}[47c-5d,50c-5d)\bs (I\cup(O(A)+47c-5d))\cup (O(A)+50c-5d).\label{thirteen}\end{equation}

Combining the sets in (\ref{one}), (\ref{two}), (\ref{three}), (\ref{four}), we conclude that $J$ covers the set \begin{equation}[7c-d,23c-2d)\bs (C\cup D\cup E).\label{fourteen}\end{equation}

Combining the sets in (\ref{five}), (\ref{six}), (\ref{seven}), (\ref{eight}), we conclude that $J$ covers \begin{equation}[23c-3d,33c-3d)\bs (F\cup (O(A)+23c-3d)).\label{fifteen}\end{equation}

Combining the sets in (\ref{nine}), (\ref{ten}), (\ref{eleven}), (\ref{twelve}), (\ref{thirteen}), we conclude that $J$ covers \begin{equation}[33c-4d,48c-5d)\bs (G\cup H\cup I\cup (O(A)+33c-4d)).\label{sixteen}\end{equation}

The largest element of $O(A)+23c-3d$ is $\om+23c-3d<23c-2d$, where we used the assumption that $d>\om$. Likewise, the largest element of $O(A)+33c-4d$ is $\om+33c-4d<33c-3d$.  Hence, we can combine the sets in (\ref{fourteen}), (\ref{fifteen}), (\ref{sixteen}) into \begin{equation}[7c-d,50c-5d)\bs (C\cup D\cup E\cup F\cup G\cup H\cup I).\label{seventeen}\end{equation}In particular, $J$ covers the set $[b+8c-d,b+48c-5d]\bs J$, which is a subset of (\ref{seventeen}).

Since $J$ is 3-free and covers all elements of $[b+8c-d,b+48c-5d]\bs J$, we conclude that $J$ is a prefix of $S(A^d_k)$. Therefore, by Lemma \ref{lem:kappa}, the sequence $S(A^d_k)$ is independent.
\end{proof}

As a result of the construction given in Proposition \ref{prop:main}, we obtain the following result.

\begin{prop}
\label{prop:epsilon}
Let $S(A) =\{a_n\}$ be an \iss, with scaling factor $\al$ and repeat factor $\rho$. Then,

\begin{enumerate}
\renewcommand{\theenumi}{\alph{enumi}}
\item Suppose that $\al'$ is a triadic number with $\al\le \al'<10\al/9$. Then, $\al'$ is the scaling factor of some independent Stanley sequence.
\item Choose $\ep>0$. There exists an integer $N_\ep(A)$ such that for all $k\ge N_\ep(A)$, every integer in the interval $[3^k (1+\epsilon)\rho,3^k (10/9-\epsilon)\rho]$ is a repeat factor for some independent Stanley sequence.
\end{enumerate}
\end{prop}

\begin{proof} Set $\lb=\lb(A)$, $\om=\om(A)$, and $\kp=\kp(A)$.

(i) Clearly $\al$ itself is a scaling factor, so suppose we have $\al'>\al$. Proposition \ref{prop:main} implies that $$\al(A^d_k)=\frac{10\al(A)}{9}-\frac{d}{3^{k+2}}$$ for large enough $k$. Set $d=3^k(10\al-9\al')$, so that $$\al'=\frac{10\al}{9}-\frac{d}{3^{k+2}}.$$Note that $d$ is an integer for large $k$, because $\al'$ is a triadic number. Since $\al'<10\al/9$, the condition $d>\omega(A)$ is satisfied for $k$ large enough. Likewise, since $\al'>\al$, we have $d=t\alpha 3^k$ for some value $t<1$ independent of $k$. Hence, we can make $k$ large enough to satisfy the condition $d\le a_{2^k}-\lb=\alpha 3^k-\lb(A)$. We conclude that there exists an independent Stanley sequence for which $\al'$ is the scaling factor.

(ii) Proposition \ref{prop:main} implies that $$\rho(A^d_k)=10a_{2^k}-d,$$for large enough $k$. Thus, $\rho(A^d_k)$ may take on any integral value $\rho'$ such that $$9a_{2^k}+\lb=10a_{2^k}-(a_{2^k}-\lb)\le \rho'<10a_{2^k}-\om.$$Pick $k_\ep$ large enough so that for any $k\ge k_\ep$, \begin{align*}9a_{2^k}+\lb &=3^{k-\kp+2}\cdot \rho(A)+\lb<3^{k-\kp+2}(1+\ep)\rho\\ \text{and}\hskip .25 in 10a_{2^k}-\om &=3^{k-\kp+2}\cdot \frac{10}{9}\rho(A)-\om>3^{k-\kp+2}(1+\ep)\left(\frac{10}{9}-\ep\right)\rho,\end{align*}where we used the fact that $\rho(A)=a_{2^\kappa}=a_{2^k}/3^{k-\kappa}$. Let us take $N_\ep(A)=k_\ep-\kp+2$. Then, for each $k\ge N_\ep(A)$, every integer in the interval $[3^k (1+\epsilon)\rho,3^k (10/9-\epsilon)\rho]$ is a repeat factor for some independent Stanley sequence.
\end{proof}

We now are able to prove Theorem \ref{thm:main}.

\begin{proof}[Proof of Theorem]
(i) We apply Proposition \ref{prop:epsilon}(a) to the sequence $S(A)=S(0)$, for which $\al(A)=1$. Hence, every triadic number $\al'\in [1,10/9)$ is a valid scaling factor. Applying Proposition \ref{prop:epsilon} again, we see that every triadic number $\al''\in [1,100/81)$ is a scaling factor. Continuing in this way, we conclude that every triadic number in $[1,(10/9)^n)$ is a scaling factor, for any value of $n$. Since $(10/9)^n$ can be made arbitrarily large, we conclude that every triadic number $\al\ge 1$ is a valid scaling factor.

(ii) Pick some small $\ep>0$. We apply Proposition \ref{prop:epsilon}(b) to the sequence $S(A_1)=S(0)$, for which $\rho(A_1)=1$. For every $k\ge N_\ep(A_1)$, each integer $x$ in the interval $[3^k (1+\epsilon),3^k (10/9-\epsilon)]$ is the repeat factor of some independent sequence $S(A_x)$.

We next apply Proposition \ref{prop:epsilon}(b) to each sequence $S(A_x)$. For every $k_x\ge N_\ep(A_x)$, each integer in the interval $[3^k (1+\epsilon)x,3^k (10/9-\epsilon)x]$ is a repeat factor. These intervals overlap as $x$ varies over the integers in $[3^k (1+\epsilon),3^k (10/9-\epsilon)]$. Hence, for every $k\ge \max_x N_\ep(A_x)$, each integer $y$ in the interval $[3^k (1+\epsilon)^2,3^k (10/9-\epsilon)^2]$ is the repeat factor of some independent sequence $S(A_y)$.

We may now apply Proposition \ref{prop:epsilon}(b) to each sequence $S(A_y)$. Continuing in this manner, we conclude that, for each $n$, there exists $N_n$ such that the following property holds: Whenever $k\ge N_n$, every integer in the interval $[3^k (1+\epsilon)^n,3^k (10/9-\epsilon)^n]$ is the repeat factor of some independent sequence. For $\ep$ small, we can pick $n$ so that $(10/9-\epsilon)^n>3\cdot (1+\epsilon)^n$. Then, $3^k (10/9-\epsilon)^n>3^{k+1}(1+\epsilon)^n$, so every sufficiently large integer is contained in a set of the form $[3^k (1+\epsilon)^n,3^k (10/9-\epsilon)^n]$ for some $k\ge N_n$. Therefore, every sufficiently large integer is the repeat factor of an independent Stanley sequence.
\end{proof}

\section{Open problems}

There remain many unanswered questions related to the growth of Stanley sequences. Our proof leaves open the question of which integers are not the repeat factor of any independent sequence. Rolnick additionally posed the problem of identifying which values of $\lb(A)$ are attainable.

\begin{conj}[Rolnick \cite{rolnick}]
Let $\lb$ be any integer other than $1,3,5,9,11,15$. Then, there exists an independent Stanley sequence $S(A)$ such that $\lb(A)=\lb$.
\end{conj}

Dependent Stanley sequences, which are described in \cite{rolnick}, follow Type 1 growth like independent sequences.  However, while independent sequences satisfy $a_{2^k}=\alpha\cdot 3^k$, dependent sequences satisfy $a_{2^k-\sigma}=\alpha\cdot 3^k+\beta\cdot 2^k$, where $\beta$ and $\sigma$ are constants. Rolnick conjectures that $\beta\ge 0$; further investigation is called for.

It appears very hard to prove that every Stanley sequence follows either Type 1 or Type 2 growth. Erd\H{o}s et al.~posed the weaker problem of showing that every Stanley sequence $\{a_n\}$ satisfies $a_n=\Omega(n^{1+\epsilon})$ for some $\epsilon>0$. This remains open.

\section{Acknowledgements}

This research was conducted at the University of Minnesota Duluth REU, funded by NSF Grant 1358659 and NSA Grant H98230-13-1-0273. The authors would like to thank Joseph Gallian for his advice, and Adam Hesterberg, Noah Arbesfeld, Daniel Kriz, and the other visitors and participants in the Duluth REU for their helpful discussions and valuable suggestions.

\bibliography{stanley}
\bibliographystyle{plain}

\end{document}